\documentclass{amsart}
\usepackage{amssymb}
\usepackage[cp1250]{inputenc}
\usepackage[OT4]{fontenc}
\usepackage{enumerate}
\def\CC{\mathbb{C}}
\def\RR{\mathbb{R}}

\def\NN{\mathbb{N}}

\def\aaa{\mathbb{A}}

\def\wi{\widetilde}

\def\pa{\partial}
\def\ov{\overline}

\def\su{\subset}
\def\lon{\longrightarrow}

\def\cC{{\mathcal C}}
\def\HH{{\mathcal H}}
\def\eps{\varepsilon}

\DeclareMathOperator{\re}{Re}
\DeclareMathOperator{\RE}{Re}
\DeclareMathOperator{\im}{Im}

\renewcommand{\phi}{\varphi}
\newtheorem{prop}{Proposition}

\newtheorem{tw}[prop]{Theorem}

\newtheoremstyle{rem}{}{}{}{}{\bf}{.}{ }{}
\theoremstyle{rem}

\newtheoremstyle{remm}{}{}{}{}{\bf}{.}{ }{}
\theoremstyle{remm}

\begin{document}


\title{Geometric properties of semitube domains}

\address{Institute of Mathematics, Faculty of Mathematics and Computer Science, Jagiellonian University, \L ojasiewicza 6, 30-348 Krak\'ow, Poland}
\thanks{The work is partially supported by 
grants of the Polish National Science Centre no. UMO-2011/03/B/ST1/04758 and DEC-2012/05/N/ST1/03067}

\email{{Lukasz.Kosinski}@im.uj.edu.pl}
\email{{Tomasz.Warszawski}@im.uj.edu.pl}
\email{{Wlodzimierz.Zwonek}@im.uj.edu.pl}
\author{\L ukasz Kosi\'nski}
\author{Tomasz Warszawski}
\author{W\l odzimierz Zwonek}


\subjclass[2010]{Primary: 32V20. Secondary: 32L05, 32V25}


\keywords{Semitube domains, Hartogs-Laurent domains, Bochner's theorem, multisubharmonic functions.}

\begin{abstract}
In the paper we study the geometry of semitube domains in $\mathbb C^2$. In particular, we extend the result of Burgu\'es and Dwilewicz for semitube domains dropping out the smoothness assumption. We also prove various properties of non-smooth pseudoconvex semitube domains obtaining among others a relation between pseudoconvexity of a semitube domain and the number of components of its vertical slices.

Finally, we present an example showing that there is a non-convex domain in $\mathbb C^n$ such that its image under arbitrary isometry is pseudoconvex.
\end{abstract}

\maketitle

\section{Introduction}
A theorem of Bochner states that a tube domain in $\CC^n$ is pseudoconvex if and only if it is convex. This fact may be seen as a starting point for our considerations.

In \cite{Bur-Dwi} a similar problem was considered for semitube domains --- domains that are invariant in one real direction (they were considered in $\CC^2$). Formally the {\it semitube domain $($set$)$ with the base $B$} being a domain (set) lying in $\mathbb R^3$ is defined as follows $$
S_B:=\{z\in\CC^2:(z_1,\RE z_2)\in B\},$$ which may be rewritten as $B\times \mathbb R$. The first observation that should be made is that there is no a direct analogue of Bochner theorem in the class of semitube domains ---  it follows easily from the fact that any domain $D\subset\CC$ induces a pseudoconvex domain of the form $S_{D\times(0,1)}$. However, it was recently proven by Burgu\'es and Dwilewicz that some additional requirement implies the convexity of a semitube domain. Namely, main result of \cite{Bur-Dwi} is that under additional assumption of smoothness any domain $D\subset\RR^3$ such that for any isometry $A$ of $\RR^3$ the semitube domain $S_{A(D)}=A(D)\times\RR$ is pseudoconvex must be convex. The main aim of our paper is to show this result without the smoothness assumption. The methods used in the paper are also quite different.

\begin{tw}\label{thm:1}
Let $D\su\RR^3$ be a domain such that the semitube $S_{A(D)}$ is pseudoconvex for any isometry $A$ of $\RR^3$. Then $D$ is convex.
\end{tw}

Another natural question that arises while considering semitube domains is the problem whether one could exhaust any pseudoconvex semitube domain with smooth semitube domains. This is the case as it is shown in the following theorem.

\begin{tw}\label{thm:3}
Any pseudoconvex semitube domain $G\su\CC^2$ can be exhausted by $\cC^\infty$-smooth strongly pseudoconvex semitube domains.
\end{tw}

Consider the following mapping $\pi:\CC^2\owns z\longmapsto(z_1,\exp(z_2))\in\CC^2$. Note that this mapping induces a holomorphic covering between semitube domains $S_D$ and Hartogs-Laurent domains $\pi(S_D)$. We call a domain $G\subset\mathbb C^2$ a {\it Hartogs-Laurent domain} if any non-empty fiber $\{z_2\in\CC:(z_1,z_2)\in G\}$ is some union of annuli, i.e. sets of the form $\{z_2\in\CC:r<|z_2|<R\}$ with $0\leq r<R\leq\infty$. The projection of $G$ on the first coordinate is called the {\it base} of the domain. The mapping $\pi$ induces a one-to-one correspondence between those two classes of domains as it is stated in the following proposition.

\begin{prop}\label{prop:1}
Let $\pi$ be as above. Then the function
$$S_D\longmapsto\pi(S_D)$$
gives a one-to-one correspondence between the class of all pseudoconvex semitube domains in $\mathbb C^2$ and the class of all pseudoconvex Hartogs-Laurent domains in $\mathbb C^2$.
\end{prop}
 \begin{proof}
Let the domain $S_{D}$ be pseudoconvex. Then $u:=-\log d_{S_D}\in PSH(S_D)$, where $d_G$ is the distance to the boundary of $G$. Since $u$ does not depend on $\im z_2$, the function $v$ given by the formula $v(z):=u(z_1,\log z_2)$, $z\in  \pi(S_D)$, is well-defined and plurisubharmonic on $\pi(S_D)$. Therefore, $$\wi v(z):=\max\{v(z),\|z\|,-\log|z_2|\},\quad z\in\pi(S_D),$$ is an exhaustion plurisubharmonic function for $\pi(S_D)$. The other implication is trivial.
\end{proof}

The above observation shows that there is a very natural relation between (pseudoconvex) semitube domains and (pseudoconvex) Hartogs-Laurent domains. There is a very rich literature on that class of domains (see e.g. \cite{Jar-Pfl}) which shows that many properties of pseudoconvex semitube domains may be concluded from the properties of pseudoconvex Hartogs-Laurent domains. In particular, very irregular Hartogs-Laurent domains, like the worm domains (see \cite{Die-For}) let us produce very irregular semitube pseudoconvex domains.

\section{Proofs of Theorem~\ref{thm:1} and Theorem~\ref{thm:3}}
We start with the proof of the main result of the paper.

\begin{proof}[Proof of Theorem~\ref{thm:1}]
Suppose that $D$ is not convex. The idea of the proof is the following. We find a sequence of  parallel segments of the constant length lying in the domain $D$ and such that the limit segment $I$ intersects the boundary at some inner point whereas the boundary of the limit segment lies in the domain. Then we rotate the domain $D$ so that $I$ were parallel to the $\re z_2$ axis. The image of the rotated semitube domain under $\pi$ is a pseudoconvex Hartogs-Laurent domain with a sequence of annuli lying in the domain. The pseudoconvexity of the Hartogs-Laurent domain lets us get a contradiction with the Kontinuit\"atssatz.

Let us proceed now formally. From \cite[Theorem 2.1.27]{Hor} there is a point $a\in\pa D$ and a quadratic polynomial $P$ on $\RR^3$ such that 
\begin{itemize}
\item $P(a)=0;$
\item $v:=\nabla P(a)\neq 0;$
\item $\langle v,X\rangle=0$ and $C:=-\HH P(a;X)>0$ for some $X\in\RR^3;$
\item $P(x)<0$ implies $x\in D$ for $x\in\RR^3$ near $a$.
\end{itemize}
By $\nabla$ and $\HH$ we denoted the gradient and the Hessian. One may assume that $\|v\|=1$.

For $\eps\geq 0$ and $\delta\in\RR$ such that $(\eps,\delta)\neq (0,0)$, $\eps\HH P(a;v)\leq 1$ and $4|\delta v^T\HH P(a)X|\leq 1$, we have \begin{eqnarray*}P(a-\eps v+\delta X)&=&P(a)+\langle\nabla P(a),-\eps v+\delta X\rangle+\frac12\HH P(a;-\eps v+\delta X)\\&=&-\eps+\frac12\HH P(a;-\eps v)+\frac12\HH P(a;\delta X)-\eps\delta v^T\HH P(a)X\\&\leq&-\eps+\frac12\eps^2\HH P(a;v)-\frac12C\delta^2+\frac14\eps \\&\leq&-\frac12\eps-\frac12C\delta^2+\frac14\eps<0.\end{eqnarray*} It means that $a-\eps v+\delta X\in D$ if this point is sufficiently close to $a$ (i.e. if $(\eps,\delta)$ is sufficiently close to $(0,0)$ but not equal to $(0,0)$ and $\eps\geq 0$). In particular, there exists a closed non-degenerate rectangle $R\su\RR^3$ such that $a\in\pa R\cap\pa D$, $a$ is not a vertex of $R$ and $R\setminus\{a\}\su D$. 

There is an isometry $A$ such that $A(R)=[\alpha,\beta]\times\{0\}\times[\alpha',\beta']\subset\RR^3$ and $A(a)\in\{\alpha,\beta\}\times\{0\}\times(\alpha',\beta')$ for some real numbers $\alpha<\beta$ and $\alpha'<\beta'$ (without loss of generality assume that $A(a)\in\{(\beta,0)\}\times(\alpha',\beta')$). Recall that $S_{A(D)}$ is pseudoconvex. Recall also that the Hartogs-Laurent domain $\Omega:=\pi(S_{A(D)})\subset\CC^2$ is pseudoconvex and because of the form of $A(D)$ we 
get a family of holomorphic mappings
$$
f_b(\lambda):=(b,\lambda),\quad\lambda\in\ov\aaa(e^{\alpha'},e^{\beta'}),\;b\in[\alpha,\beta], \text{ where }\aaa(p,q):=\{\lambda\in\CC:p<|\lambda|<q\},
$$
such that 
$$
\bigcup_{b\in[\alpha,\beta)}f_b(\ov\aaa(e^{\alpha^{\prime}},e^{\beta^{\prime}}))\subset\Omega,
$$

$$
\bigcup_{b\in[\alpha,\beta]}f_b(\pa\aaa(e^{\alpha'},e^{\beta'}))\su\su\Omega.
$$ 
However, $f_{\beta}(\ov\aaa(e^{\alpha^{\prime}},e^{\beta^{\prime}}))\not\subset\Omega$, which contradicts the Kontinuit\"atssatz in the form formulated in \cite[Theorem 4.1.19]{JJ}.
\end{proof}

Now we go on to the proof of Theorem~\ref{thm:3}.

\begin{proof}[Proof of Theorem~\ref{thm:3}]
Let $u:=-\log d_G\in PSH(G)$ and $G_\eps:=\{z\in G:d_G(z)>\eps\}$ for $\eps\in(0,1)$. 
Define the standard regularizations $u_\eps$ of $u$ with the help of convolution with radial functions. We have $u_\eps\in PSH\cap\cC^\infty(G_\eps)$ and $u_\eps\searrow u$ if $\eps\searrow 0$. Moreover, $u_\eps$ does not depend on $\im z_2$.

For $\eps\in(0,1)$ and $\delta>0$ define $$\wi u_\eps(z):=u_\eps(z)+\eps\|(z_1,\re z_2)\|^2,\quad\wi G_{\eps,\delta}:=\{z\in G_\eps:\wi u_\eps(z)<1/\delta\}.$$ Note that $\ov{\wi G_{\eps,\delta}}\su G_\eps$ for $\delta>-1/\log\eps$. Indeed, if $z_n\in\wi G_{\eps,\delta}$, $z_n\to z$, then $u(z_n)\leq\wi u_\eps(z_n)<1/\delta<-\log\eps$, so $u(z)<-\log\eps$.

By the Sard Theorem for any $\eps>0$ the set $A_\eps$ of $\delta>0$ such that $\nabla\wi u_\eps(z)\neq 0$ if $\wi u_\eps(z)=1/\delta$ is dense in $\RR_+$. For $n\in\NN$ we choose a number $\delta_{1/n}$ such that 
\begin{itemize}
\item $\delta_{1/n}>-1/\log(1/n);$
\item $\delta_{1/n}\in A_{1/n}$.
\end{itemize} 
Since the minorants $-1/\log(1/n)$ tend to zero, one may assume additionally that $\delta_{1/n}\searrow 0$ as $n\nearrow\infty$. Then we define $$\wi G_{1/n}:=\wi G_{1/n,\delta_{1/n}}.$$

From the following properties 
\begin{itemize}
\item $\wi u_{1/n}-1/\delta_{1/n}$ are $\cC^\infty$-smooth strongly plurisubharmonic defining functions of $\wi G_{1/n};$
\item $\wi u_\eps$ are independent on $\im z_2$
\end{itemize}
it follows that $\wi G_{1/n}$ are $\cC^\infty$-smooth strongly pseudoconvex semitube open sets. We directly check that $\wi G_{1/n}\su\wi G_{1/m}\su G$ if $n<m$ and any $z\in G$ belongs to some $\wi G_{1/n}$.

Finally, we fix $z\in G$ and define $G_n$ as the component of $\wi G_{1/n}$ containing $z$. Then $G_n\su G_{n+1}\su G$ and $\bigcup_n G_n=G$ (indeed, let $x\in G$, take a curve $\gamma\su G$ joining $x$ and $z$, then $\gamma\su\wi G_{1/n_1}\cup\ldots\cup\wi G_{1/n_m}=\wi G_{1/\max n_k}$ and $x\in\gamma\su G_{\max n_k}$). 
\end{proof}

Remark that if $G=S_D$, where $D\subset\mathbb R^3$, then it follows from the construction of the objects in the proof of the above result that $S_{A(\mathcal P(G_n))}$, where $\mathcal P:\RR^4\lon\RR^3$ is the projection, are strongly pseudoconvex domains exhausting the domain $S_{A(D)}$ for any isometry of $\mathbb R^3$. Thus Theorem~\ref{thm:1} follows from the same result for the strongly pseudoconvex case as it is done in \cite{Bur-Dwi}. However, it seems to us that the proof of Theorem~\ref{thm:1} presented by us is simpler and more self-contained.

\section{More problems related to semitube domains}

Note that the reasoning used in the proof of Theorem~\ref{thm:1} also implies the following property of pseudoconvex Hartogs-Laurent and semitube domains.

\begin{prop}\label{prop:semicontinuity} Let $G\subset\CC^2$ be a pseudoconvex Hartogs-Laurent domain with the base $\Omega\subset\CC$. Consider the function
$$t:\Omega\owns z\longmapsto\text{number of components of $G_z$},$$
where $G_z:=G\cap(\{z\}\times\CC)$.
Then $t$ is lower semicontinuous.

Consequently, if $D\subset\RR^3$ is such that $S_D$ is a pseudoconvex semitube domain then the function
$$
s:D_1\owns z\longmapsto\text{number of components of $D\cap(\{z\}\times\RR)$},
$$
where $D_1:=\{z\in\CC:D\cap (\{z\}\times\RR)\neq\emptyset\}$, is lower semicontinuous.
\end{prop}
\begin{proof}
Fix $z_0\in\Omega$. Let $w_1,\ldots,w_k\in G_{z_0}$ be points from different components of
$G_{z_0}$. Now making use of the Kontinuit\"atssatz for the annuli (as in the proof of the previous theorem)
we easily get that for $z\in \Omega$ sufficiently close to $z_0$ the number of components of $G_z$ is at least $k$ which finishes the proof.

The case of semitube domains follows from the case of the Hartogs-Laurent domains by applying the result for the domain $\pi(S_D)$.
\end{proof}

Note that the above property easily implies that the semitube domain over the torus in a `vertical position' (and many other) as described in Section 6.4 of \cite{Bur-Dwi} is not pseudoconvex.

\bigskip
 
In view of Theorem~\ref{thm:1} it would also be interesting and natural to consider the following problem.
Let $D\subset\CC^n$ be a domain satisfying the following condition. For any real isometry $A$ of $\CC^n=\RR^{2n}$ the set $A(D)$ is pseudoconvex. Does it follow that $D$ is convex? Certainly the problem is non-trivial for $n\geq 2$. We shall show below that the answer is negative for $n\geq 2$, too.

\begin{prop}\label{prop:example} Let $n\ge 2$. Then there is a non-convex domain $D\subset\CC^n$ such that $A(D)$ is pseudoconvex for any real isometry of $\mathbb C^n=\mathbb R^{2n}$.
\end{prop}
\begin{proof}
At first consider a class of functions defined on domains $\Omega\su\RR^m$, $m\geq 2$. We call an upper semicontinuous function $u:\Omega\lon [-\infty,\infty)$ {\it multisubharmonic} if $u$ restricted to $\Omega\cap(L+a)$ is subharmonic for any two-dimensional subspace $L\subset\RR^m$ and a point $a\in\RR^m$ such that $\Omega\cap(L+a)\neq\emptyset$.
Let us make the last statement precise --- the function $u$ on $\Omega\cap(L+a)$ is considered to be subharmonic if
for some (any) pair of vectors $X$ and $Y$ forming an orthonormal basis of $L$ the function
$(t,s)\longmapsto u(a+tX+sY)$ is subharmonic on its domain (lying in $\RR^2$). Certainly, in the case of $u$ being $\cC^2$ we have the following simple description:
$$\Delta_{X,Y}u(a):=\frac{\partial^2 u}{\partial X^2}(a)+\frac{\partial^2 u}{\partial Y^2}(a)\geq 0$$
for any $X,Y\in\RR^m$, $||X||=||Y||=1$, $\langle X,Y \rangle=0$ and $a\in \Omega$.

It is clear that any multisubharmonic function (in $\CC^n=\RR^{2n}$) is plurisubharmonic and these two concepts are the same in $\CC$.

For $m\geq 2$ and $\alpha\in(0,1]$ consider the following function $$u(x):=\frac 12(x_1^2+\ldots+x_{m-1}^2-\alpha x_m^2).$$ We have 
$$\Delta_{X,Y}u(a)=X_1^2+\ldots+X_{m-1}^2-\alpha X_m^2+Y_1^2+\ldots+Y_{m-1}^2-\alpha Y_m^2.$$
Then for any orthonormal $X,Y$ we get $\Delta_{X,Y}u(a)=2-(1+\alpha)(X_m^2+Y_m^2)$. And now note that
$$(1-X_m^2)(1-Y_m^2)=(X_1^2+\ldots+X_{m-1}^2)(Y_1^2+\ldots+Y_{m-1}^2)$$$$\geq(X_1Y_1+\ldots+X_{m-1}Y_{m-1})^2=X_m^2Y_m^2,$$ whence $X_m^2+Y_m^2\leq 1$ and $\Delta_{X,Y}u(a)\geq 1-\alpha$, so $u$ is multisubharmonic.

Now define the set $$D:=\{z\in\CC^n:u(z)<1\}\quad(m:=2n).$$ Note that $D$ is connected and non-convex. Then it follows from the multisubharmonicity of $u$ that $A(D)$ is pseudoconvex for any real isometry $A$.
\end{proof}

\end{document}